\documentclass{article}

\usepackage[margin = 1in]{geometry}

\usepackage{titlesec} 

\usepackage{tikz}
\usepackage{tikz-cd}
\usetikzlibrary{cd, arrows.meta, calc}
\usepackage{url, hyperref}
\usepackage{float}

\tikzset{->-/.style={decoration={
  markings,
  mark=at position .45 with {\arrow{>}}},postaction={decorate}}}
\usepackage{amsmath}
\usepackage{stmaryrd}
\usepackage{amssymb}
\usepackage{enumitem}
\usepackage{graphicx}
\usepackage{mathdots}
\usepackage{color}
\usepackage{diagbox}
\usepackage{array, makecell}
\usepackage{rotating}
\usepackage{amsthm}
\usepackage{dsfont}
\usepackage{adjustbox}
\usetikzlibrary{arrows}
\usepackage{mathtools}

\def\cM{{\overline{\mathcal{M}}}}

\def\cW{\mathcal{W}}

\def\oM{\overline{\mathcal{M}}}
\def\cM{{\mathcal{M}}}

\def\Z{\mathbb{Z}}

\def\R{\mathbb{R}}

\def\b1{{\bf 1}}

\def\Aut{{\rm Aut}}

\def\P{\mathbb{P}}

\newcount\colveccount
\newcommand*\colvec[1]{
        \global\colveccount#1
        \begin{pmatrix}
        \colvecnext
}
\def\colvecnext#1{
        #1
        \global\advance\colveccount-1
        \ifnum\colveccount>0
                \\
                \expandafter\colvecnext
        \else
                \end{pmatrix}
        \fi
}

\usepackage{amsthm}
\usepackage{esvect}
\newtheorem{definition}{Definition}[section]
\newtheorem{theorem}[definition]{Theorem}

\newtheorem{proposition}[definition]{Proposition}

\newtheorem{lemma}[definition]{Lemma}
\newtheorem{remark}[definition]{Remark}

\usepackage{amsmath,calligra,mathrsfs}
\DeclareMathOperator{\cHom}{\mathscr{H}\text{\kern -3pt {\calligra\large om}}\,}

\usepackage{blkarray}
\usepackage{booktabs}
\usepackage{verbatim}
\usepackage{subcaption}

\newcommand{\tropC}{\mathsf{C}}
\newcommand{\defic}[1]{\operatorname{def}(#1)}
\newcommand{\ov}[1]{\operatorname{ov}(#1)}
\newcommand{\val}[1]{\operatorname{val}(#1)}

\newcommand{\W}{\mathsf{W}}

\newcommand{\nn}{n^{\operatorname{Aut}}}

\definecolor{lightgreen}{HTML}{E8F9E2}
\definecolor{lightblue}{HTML}{ADD8E6}
\definecolor{lightcyan}{HTML}{E0FFFF}
\definecolor{lightred}{HTML}{ffcccb}
\definecolor{lightpink}{HTML}{ffe6e6}
\definecolor{coralpink}{HTML}{F88379}
\definecolor{burntsiena}{HTML}{E97451}
\definecolor{forestgreen}{HTML}{228B22}
\definecolor{skyblue}{HTML}{448EE4}
\definecolor{darkgreen}{HTML}{006400}
\definecolor{darkblue}{HTML}{00008B}
\definecolor{darkorchid}{HTML}{9932CC}

\title{Tropical and algebraic elliptic plane curves with fixed $j$-invariant}

\author{Alessio Cela and Sae Koyama}

\date{ }

\usepackage{graphicx}

\begin{document}

\maketitle

\begin{abstract}

We tropically enumerate planar well-spaced elliptic curves with a fixed $j$-invariant. Combined with the recent genus 1 correspondence theorem of Cela--Koyama, this yields an alternative proof of Pandharipande's algebraic enumeration.

\end{abstract}
\tableofcontents

\section{Introduction}

Pandharipande \cite{Pandh97} computed via algebraic methods the number $E_{d,j}$ of degree $d$ projective plane curves through $3d-1$ points with a fixed $j$-invariant. For $j \neq 0, 1728$:
$$
E_{d,j} = \binom{d-1}{2} N_d
$$
For $j \in \{0, 1728\}$, this count is adjusted by a factor of $1/3$ and $1/2$, respectively, due to extra automorphisms. Partial prior results were also obtained in \cite{Alu, DF}. Here $N_d$ denotes the number of degree $d$ planar rational curves through $3d-1$ general points computed recursively by the celebrated Kontsevich's formula \cite{Kont}. See also \cite{GM08} for a tropical proof. 

The purpose of this paper is to prove Pandharipande's formula tropically, correcting an existing oversight in the literature \cite{KM09, LR18}. Kerber--Markwig \cite{KM09} constructed a weighted polyhedral complex that parametrizes certain tropical elliptic curves. They showed that the corresponding weighted enumeration of tropical elliptic curves with a fixed $j$-invariant is well-defined and matches Pandharipande's formula. Their complex is formed from the moduli space of tropical curves in $\P^2$ by discarding cones of dimension bigger than expected and assigning appropriate multiplicities to the remaining maximal cones. As explained by Cela-Koyama \cite{CK}, neither the surviving cones nor their assigned weights directly reflect the actual counts of algebraic curves lifting to those tropical types. Rather, the multiplicities were chosen in an ad hoc manner to ensure that the corresponding invariant is independent of the choice of $j$ and the interpolation points. Consequently, without directly inspecting the two formulas, there is no obvious structural reason why these two counts should coincide, nor can one deduce Pandharipande's enumeration from the tropical one. 

Building on the moduli space of well-spaced tropical curves introduced in \cite{Tor14, RSWII}, Cela–Koyama derived an explicit combinatorial formula for the automorphism-weighted count of algebraic curves lifting maximally degenerate well-spaced elliptic tropical curves. This equips the corresponding cone complex of well-spaced tropical curves with geometric weights, yielding a tropical correspondence theorem for elliptic curves in \textit{any} toric variety. In the present paper, we specialize their correspondence theorem to dimension two, specifically for $\mathbb{P}^2$. When the cycle of the tropical curve does not span all of $\mathbb{R}^2$, the multiplicities appearing in Cela--Koyama \cite{CK} differ from those used by Kerber--Markwig \cite{KM09}. In fact, even the tropical curves appearing in the two cone complexes are different. 

We remark that Len–Ranganathan \cite{LR18} also constructed a weighted polyhedral cone complex parametrizing certain genus-1 tropical curves and established a correspondence theorem for elliptic curves on toric surfaces, which they subsequently used to enumerate elliptic curves with a fixed $j$-invariant on Hirzebruch surfaces. However, their cone complex differs from the cone complex of well-spaced tropical curves in \cite{CK}, notably in the presence or absence of certain cones. This may create some inaccuracies in certain curve counts in toric surfaces. The work of \cite{LR18} predated the formal introduction of the moduli stack of well-spaced curves by Ranganathan, Santos-Parker, and Wise \cite{RSWII}, which serves as the starting point for the Cela--Koyama correspondence theorem \cite{CK}. Tables \ref{tab: comparing weights deficiency 1} and \ref{tab: comparing weights deficiency 2} (adapted from \cite{CK}) illustrate some of these differences.  

\begin{table}[H]
    \centering
    \renewcommand{\arraystretch}{1.5}
    \begin{tabular}{@{}lcc@{}}
        \toprule
        &
        \begin{tikzpicture}[scale=0.8]
            \fill[white] (0,1.2) circle[radius=2pt];
            \fill (0,0) circle[radius=2pt];
            \fill (2,0) circle[radius=2pt];
            \fill (1.5,0) circle[radius=2pt];
            \draw (0, 0.05) -- (1.5, 0.05);
            \draw (0, -0.05) -- (1.5, -0.05);
            \draw (1.5, 0) -- (2,0);
            \draw (0, 0) -- (-1, 1); 
            \draw (0, 0) -- (-1, -1); 
            \draw (2, 0) -- (3, 1); 
            \draw (2, 0) -- (3, -1);

            \draw[forestgreen, ->] (0,0.15) -- (-0.5, 0.65);
            \node[forestgreen] at (-0.5,0.8) {$u$};

            \draw[forestgreen, ->] (0.2, 0.15) -- (1, 0.15); 
            \node[forestgreen] at (0.6, 0.35) {$a n_1$};

            \draw[forestgreen, ->] (0.2, -0.15) -- (1, -0.15);
            \node[forestgreen] at (0.6, -0.4) {$b n_1$};
            
            \node at (-0.4,0) {$V_1$};
        \end{tikzpicture} 
        &
        \begin{tikzpicture}[scale=0.8]
            \fill (0,0) circle[radius=2pt];
            \fill (2,0) circle[radius=2pt];
            \fill (0.5,0) circle[radius=2pt];
            \fill (1.5,0) circle[radius=2pt];
            \draw (0, 0) -- (0.5, 0);
            \draw (0.5, 0.05) -- (1.5, 0.05); 
            \draw (0.5, -0.05) -- (1.5, -0.05); 
            \draw (1.5, 0) -- (2, 0); 
            \draw (0, 0) -- (-1, 1); 
            \draw (0, 0) -- (-1, -1); 
            \draw (2, 0) -- (3, 1); 
            \draw (2, 0) -- (3, -1);
            \node at (0.3,0.3) {$l$};
            \node at (1.7,0.3) {$l$};

            \draw[forestgreen, ->] (0.6, 0.15) -- (1.4, 0.15); 
            \node[forestgreen] at (1, 0.35) {$a n_1$};

            \draw[forestgreen, ->] (0.6, -0.15) -- (1.4, -0.15);
            \node[forestgreen] at (1, -0.4) {$b n_1$};
        \end{tikzpicture} 
        \\
        \midrule
        multiplicities in \cite{CK} & $\gcd(a,b) \times (|\det(u, n_1)| - 1)/\Aut(\sigma').$ & $\gcd(a,b)/\Aut(\sigma').$ \\
        multiplicities in \cite{KM09} & $\gcd(a,b) \times |\det(u, n_1)|/\Aut(\sigma').$ & 0 \\
        multiplicities in \cite{LR18} & not a cone in the cone complex of tropical curves & $\gcd(a,b)/\Aut(\sigma').$ \\
        \bottomrule
    \end{tabular}
    \vspace{0.2cm}
    \caption{Comparing multiplicities of curves with deficiency $1$.}
    \label{tab: comparing weights deficiency 1}
\end{table}

\begin{table}[H]
    \centering
    \renewcommand{\arraystretch}{1.5}
    \begin{tabular}{@{}lcc@{}}
        \toprule
        &
        \begin{tikzpicture}[scale=0.8]
            \fill[white] (0,1.2) circle[radius=2pt];
            \fill (0,0) circle[radius=2pt];
            \fill (2,0) circle[radius=2pt];
            \draw (0, 0) -- (2,0);
            \draw (0, 0) -- (-1, 1); 
            \draw (0, 0) -- (-1, -1); 
            \draw (2, 0) -- (3, 1); 
            \draw (2, 0) -- (3, -1);
            \draw[dashed] (0,0)  to[in=50,out=130,loop, style={min distance=12mm}] (0,0);

            \node at (-0.4,0) {$V_1$};
        \end{tikzpicture} 
        &
        \begin{tikzpicture}[scale=0.8]
            \fill (0,0) circle[radius=2pt];
            \fill (2,0) circle[radius=2pt];
            \draw (0, 0) -- (2,0);
            \draw (1.5, 0) -- (2, 0); 
            \draw (0, 0) -- (-1, 1); 
            \draw (0, 0) -- (-1, -1); 
            \draw (2, 0) -- (3, 1); 
            \draw (2, 0) -- (3, -1);
            \draw[dashed] (1,0)  to[in=50,out=130,loop, style={min distance=12mm}] (1,0);
            \node at (0.3,0.3) {$l$};
            \node at (1.7,0.3) {$l$};

            \draw[forestgreen, ->] (1,0) -- (0.3, 0);
            \draw[forestgreen, ->] (1,0) -- (1.7, 0);

            \node[forestgreen] at (1.4, -0.3) {$d n_1$};

            \fill (1,0) circle[radius=2pt];
        \end{tikzpicture} 
        \\
        \midrule
        multiplicities in \cite{CK}  & $i(\text{dual polygon at $V_1$})$ & $(d-1)/2$ \\
        multiplicities in \cite{KM09}  & $\text{Area}(\text{dual polygon at } V_1) - \frac{1}{2}$ & 0 \\
        multiplicities in \cite{LR18}   & $i(\text{dual polygon at $V_1$})$ & $(d-1)/2$ \\
        \bottomrule
    \end{tabular}
    \vspace{0.2cm}
    \caption{Comparing multiplicities of curves with deficiency $2$.}
    \label{tab: comparing weights deficiency 2}
\end{table}

Ideally, a tropical enumeration should correspond to the algebraic one. In this paper we recover Kerber-Markwig's formulas using the algebraically grounded multiplicities from \cite{CK}. 

The main theorem of this paper is the following.

\begin{theorem}\label{thm: main}
    For general points in $\P^2$ and general $j$-invariant, we have
    \begin{equation}\label{eqn: main identity}
    E_{d,j}= E_{d,j}^{\mathrm{trop}}= \binom{d-1}{2}N_d 
    \end{equation}
    Moreover, one has 
    \begin{equation}\label{eqn: main-identity-2}
    E_{d,j}^{\mathrm{trop}}= \sum_{\tropC} \bigg( \sum_T (2\mathrm{Area}(T)^2- \frac{1}{2}) \mathrm{mult}(\tropC) \bigg) 
    \end{equation}
    where $\tropC$ ranges over all rational tropical curves passing through the $3d - 1$ points, $T$ ranges over all triangles in the Newton subdivision dual to $\tropC$ and $\mathrm{mult}(\tropC)$ denotes Mikhalkin's multiplicity of $\tropC$.
\end{theorem}

The equality in Equation~\eqref{eqn: main-identity-2} of Theorem~\ref{thm: main} was established by Kerber and Markwig~\cite{KM09} using non-algebraic multiplicities and working with $j \in \mathbb{R}_{\geq 0}$ very small. From this, they deduced a surprising closed formula expressing $N_d$ entirely in terms of certain path counts~\cite[Corollary~7.1]{KM09}. This justifies the importance of the appearance of this formula in the above statement. Note that when $j$ is very small, both curves in Table~\ref{tab: comparing weights deficiency 1} contribute to the computation and cannot be omitted, as in the cone complex of Len–Ranganathan~\cite{LR18}. 

One of the main inputs in \cite{KM09} used to prove Theorem~\ref{thm: main} was the fact that for large $j$, all enumerated tropical curves have a contracted cycle. This remains true in our setting (see Proposition \ref{prop: contracted edge}); however, because we work with a modified cone complex of tropical curves, the argument from \cite{KM09} does not apply verbatim. See also Remark \ref{rmk: pb-string}.

The proof of Theorem \ref{thm: main} is presented in \S\ref{sec: j large} and \S\ref{sec: j small} and is based on the recent tropical correspondence established in \cite[Theorems A and B]{CK}. The enumerative count of elliptic curves passing through $3d-1$ general points in $\mathbb{P}^2$ with a fixed general $j$-invariant is obtained in this paper as an invariant over the moduli space of well-spaced elliptic curves in $\mathbb{P}^2$ \cite{RSWII}. By \cite[Proposition 1.2]{CK}, this invariants are always integer. 

In fact, the count we are interested in in this paper is the genuinely enumerative one. This differs from the ordinary Gromov-Witten count as the next theorem shows.

\begin{theorem}\label{thm:GW-count}
    The Gromov-Witten count of degree $d$ planar elliptic curves in $\P^2$ through $3d-1$ points and with fixed $j$ invariant is
    $$
    E_{d,j}^{\mathrm{vir}}=d^2 N_d.
    $$
\end{theorem}

The proof of Theorem \ref{thm:GW-count} is presented in \S\ref{sec: GW-count}.

\section{Acknowledgments}

We are indebted to Hannah Markwig, who explained the content of her paper \cite{KM09} to us several years ago. She also pointed out that her multiplicities were explicitly designed to yield a well-defined count, although they do not correspond to algebraic curves.

We are also grateful to Dhruv Ranganathan for discussions related to the content of this paper and its relation to his work \cite{LR18}. 

A.~C.\ is supported by SNF grant P500PT-222363. S.~K.'s doctoral programme is funded by St John's College, Cambridge.

\section{Algebraic counts}

Let $\oM_{1,3d-1}(\P^2,d)$ be the moduli stack of $3d-1$-pointed stable maps from genus $1$ curves to $\P^2$ of degree $d$. Consider the morphism
\begin{equation}\label{eqn: map}
\tau= \mathrm{ev} \times \mathsf{st}: \oM_{1,3d-1}(\P^2,d) \to (\P^2)^{3d-1} \times \oM_{1,1}
\end{equation}
given by the evaluation map and the stabilized domain curve with the $1$-st marking. The expected relative dimension of $\tau$ is $0$. However, because $\oM_{1,3d-1}(\P^2,d)$ has several components of different dimensions, the degree of $\tau$ is not well-defined. There are at least two possible approaches to obtain a curve count:
\begin{enumerate}
    \item[$\bullet$] (Gromov-Witten theory) The space $\oM_{1,3d-1}(\P^2,d)$ carries a virtual fundamental class 
    $$
    [\oM_{1,3d-1}(\P^2,d)]^{\mathrm{vir}} \in \mathsf{CH}_{\mathrm{vdim}}(\oM_{1,3d-1}(\P^2,d))
    $$   
    where $\mathrm{vdim}=3d+(3d-1)=6d-1$, and 
    $$
    E_{d,j}^{\mathrm{vir}}\coloneqq \int_{[\oM_{1,3d-1}(\P^2,d)]^{\mathrm{vir}}} \tau^*(\mathsf{pt})
    $$
    is the virtual count of elliptic curves in $\P^2$ passing through $3d-1$ points with a fixed $j$-invariant.
    \item[$\bullet$] (Enumerative count) The open locus $\cM_{1,3d-1}(\P^2,d) \subseteq \oM_{1,3d-1}(\P^2,d)$ is pure-dimensional, and thus the restriction of $\tau$ to this locus has a well-defined degree $E_{d,j}$. This is precisely the enumerative count.
\end{enumerate}

It follows from Theorems \ref{thm: main} and \ref{thm:GW-count} that $E_{d,j}^{\mathrm{vir}} \neq E_{d,j}$.

\subsection{Gromov-Witten count}\label{sec: GW-count}

In this section, we prove Theorem \ref{thm:GW-count}.

\begin{proof}[Proof of Theorem \ref{thm:GW-count}]
    Let $\xi: \overline{M}_{0,3} \to \oM_{1,1}$ be the map given by gluing the marked curve $(\P^1,0,1,\infty)$ at $1$ and $\infty$. Let $\mathsf{H}$ denote the hyperplane class in $\P^2$.  By the splitting axiom in Gromov-Witten theory, writing $\Delta_{\P^2}= 1 \otimes \mathsf{pt} + \mathsf{pt} \otimes 1 + \mathsf{H} \otimes \mathsf{H}$, we have
\begin{align*}
    E_{d,j}^{\mathrm{vir}} &= \xi^! \mathsf{st}_*( \mathrm{ev}^*(\mathsf{pt}^{3d-1} \cap [\oM_{1,n}(\P^2,d)]^{\mathrm{vir}})) \\
    &= \langle \mathsf{pt}^{3d-1}, 1, \mathsf{pt} \rangle^{\P^2}_{0,3d+1} + \langle \mathsf{pt}^{3d-1}, \mathsf{pt}, 1 \rangle^{\P^2}_{0,3d+1} + \langle \mathsf{pt}^{3d-1}, \mathsf{H}, \mathsf{H} \rangle^{\P^2}_{0,3d+1} \\
    &= d^2 N_d,
\end{align*}
where $\langle \mathsf{pt}^{3d-1},-, - \rangle^{\P^2}_{0,3d+1}$ is the genus $0$ Gromov-Witten invariant of $\P^2$ with $3d+1$ marked points and degree $d$. In the third equality, we used the fact that the first two summands vanish by the fundamental class axiom. For the third term, we applied the divisor equation twice and used the fact that genus $0$ Gromov-Witten invariants of $\P^2$ are enumerative, and thus $N_d = \langle \mathsf{pt}^{3d-1} \rangle^{\P^2}_{0,3d-1}$.
\end{proof}

\subsection{Enumerative count and well-space maps}

For a smooth projective toric variety $X$ of dimension $r$ over $\mathbb{C}$, Ranganathan--Santos-Parker--Wise~\cite{RSWI, RSWII} introduced the moduli space $\mathcal{W}_{\Gamma}(X)$ of genus $1$ well-spaced stable maps to $X$ with $n$ free markings and prescribed contact orders along the toric boundary $\partial X$.

The curve class $A \in H_2(X)$ is fixed by the contact orders, and the domain curve carries $n+m$ total markings: $m$ boundary-marked points $z_j$ with contact orders $\delta_j \in N_X$, and $n$ free points $p_i$ with contact order $0$. The tuple $\Gamma$ encodes this complete discrete contact and marking data.

When $X=\mathbb{P}^2$, the maximal tangency condition for degree $d$ curves corresponds to the tuple of vectors $\delta_d$ containing $d$ copies of the primitive generator for each ray of the fan $\Sigma(\mathbb{P}^2)$, that is,
\[
\delta_d = ( \underbrace{e_1, \dots, e_1}_d, \underbrace{e_2, \dots, e_2}_d, \underbrace{-e_1-e_2, \dots, -e_1-e_2}_d ).
\]

We write
\[
\Gamma_d=(1,3d-1,3d,\delta_d)
\]
for the discrete data consisting of genus \(1\), \(3d-1\) free markings,
\(3d\) degree markings, and degree \(\delta_d\). 

Let $\varphi: \widetilde{\mathcal{M}}_{1,3d-1}(\mathbb{P}^2,d) \to \mathcal{M}_{1,3d-1}(\mathbb{P}^2,d)^{\circ}$ be the degree-$(d!)^3$ étale cover obtained by ordering the contact points. Here, $\mathcal{M}_{1,3d-1}(\mathbb{P}^2,d)^{\circ}$ denotes the open, dense locus of maps whose domain curve is smooth and intersects each of the three toric divisors at $d$ distinct points, with no point belonging to more than one divisor. The moduli stack $\mathcal{W}_{\Gamma_d}(\P^2)$ furnishes a proper, logarithmically smooth and pure of dimensional compactification of the stack of $\widetilde{\cM}_{1,3d-1}(\P^2,d)$

Because the exact moduli functor and boundary points are not required for this paper, we omit their detailed definitions here. See \cite{RSWI,CK} for the details.

The dimension of $\cW_{\Gamma_d}(\P^2)$ is $6d-1$ and the map $\tau|_{\cM_{1,3d-1}(\P^2,d)^{\circ}}\circ \varphi$ extends to 
$$
\cW_{\Gamma_d}(\P^2) \to (\P^2)^n \times \oM_{1,1}.
$$
The degree of this last map can be computed tropically as we explain in \S\ref{sec: tropical-preamble}.

\begin{remark}
    Since we are unconcerned with the order of the contact points $z_j$, we divide by $(d!)^3$ to forget this ordering in our computations. That is, $E_{d,j}$ equals the degree of the map $\mathcal{W}_{\Gamma_d}(\mathbb{P}^2) \to (\mathbb{P}^2)^{3d-1} \times \overline{\mathcal{M}}_{1,1}$ divided by $(d!)^3$
\end{remark}

\section{Tropical count}\label{sec: tropical-preamble}

In this section, we explain the correspondence theorem in \cite{CK}.

\subsection{Tropical curves in $\R^2$}

We briefly introduce tropical curves in $\R^2$ and their moduli space. We adopt the conventions of \cite{CK} and refer to it for further details.

A tropical curve is a tuple
\[
\tropC=(G,g,p,z,\ell),
\]
where \(G\) is a finite connected graph, \(g\colon V(G)\to\mathbb{Z}_{\geq 0}\)
is a genus function,
\[
p\colon\{1,\ldots,n\}\longrightarrow L^{\mathrm{e}}(G),
\qquad
z\colon\{1,\ldots,m\}\longrightarrow L^{\mathrm{e}}(G)
\]
label the unbounded ends $L^e(G)$, and
\(\ell\colon E(G)\to\mathbb{R}_{>0}\) assigns a length to every bounded
edge.  Its genus is
\[
g(C)=b_1(G)+\sum_{V\in V(G)}g(V).
\]

Let $\delta=( \delta_j )_{j=1,\ldots,m}$ be a tuple of vectors in $\Z^2$ summing up to $0$. A parametrized tropical curve of degree $\delta$ is a continuous
piecewise integral-affine map
\[
h\colon \tropC\longrightarrow \R^2
\]
such that, for every half-edge \(\zeta\) adjacent to a vertex \(V\), the
restriction of \(h\) to \(\zeta\) has an integral slope
\(u_\zeta\in \Z^2\), and the following conditions hold:
\begin{enumerate}
    \item[$\bullet$] for every vertex \(V\),
    \[
    \sum_{\zeta\ni V}u_\zeta=0;
    \]
    \item[$\bullet$] the free markings are contracted,
    \[
    u_{p_i}=0,
    \qquad i=1,\ldots,n;
    \]
    \item[$\bullet$] the degree markings have the prescribed directions,
    \[
    u_{z_j}=\delta_j,
    \qquad j=1,\ldots,m.
    \]
\end{enumerate}
A parametrized tropical curve is \emph{stable} if every vertex of genus 0 is at least trivalent. We denote by $\Gamma = (g, n, m, \delta)$ the data consisting of the genus $g$, the number $n$ of free markings, the number of contact markings and the contact orders $\delta$. We write
\[
M^{\mathrm{trop}}_{\Gamma}(\mathbb{R}^2)
\]
for the generalized cone complex parametrizing isomorphism classes of stable parametrized tropical curves with prescribed discrete data $\Gamma$. The cones of this complex correspond to fixed combinatorial types, each given by a decorated graph $(G, g, p, z)$ together with the direction vector $u_\zeta$ for every half-edge~$\zeta$.

In this paper, we will only be interested in the situation where $\Gamma$ is $\Gamma_d=(1,3d-1,3d, \delta_d)$ or $\Gamma_d'=(0,3d-1,3d,\delta_d)$.

\subsection{Elliptic well-spaces tropical curves in $\R^2$}

\subsubsection{Main definitions}

Let $\mathsf{C}$ be a genus $1$ tropical curve and let \(\tropC_0\subseteq \tropC\) be the \emph{core} of \(\tropC\), namely the minimal
connected subgraph of genus \(1\).  When all vertex genera vanish,
\(\tropC_0\) is the unique circuit of \(\tropC\).  For vertices \(V,W\in V(\tropC)\), write
\[
V<W
\quad\Longleftrightarrow\quad
\operatorname{dist}(V,\tropC_0)
<
\operatorname{dist}(W,\tropC_0).
\]

\begin{definition}
A \emph{radial alignment} on \(C\) is a surjective map
\[
\rho\colon V(C)\longrightarrow\{0,\ldots,k\}
\]
for some \(k\geq 0\), such that
\[
\rho^{-1}(0)=V(\tropC_0)
\]
and
\[
V<W \quad\Longrightarrow\quad \rho(V)<\rho(W).
\]
In particular, vertices assigned the same value by \(\rho\) have the
same distance from \(\tropC_0\).
\end{definition}

Adding a radial alignment to the combinatorial type produces a
subdivision
\[
M^{\mathrm{trop, rad}}_{\Gamma}(\R^2)
\longrightarrow
M^{\mathrm{trop}}_{\Gamma}(\R^2).
\]
The cones of
\(M^{\mathrm{trop, rad}}_{\Gamma}(\R^2)\)
correspond to combinatorial types together with a radial alignment.

To formulate well-spacedness, let
\[
h_0\colon\mathring \tropC_0\longrightarrow \R^2
\]
be the \emph{neighborhood of the core}.  Here \(\mathring \tropC_0\) is
obtained by taking all vertices and edges of \(\tropC_0\), together with every
half-edge adjacent to a vertex of \(\tropC_0\); edges leaving \(\tropC_0\) are
replaced by unbounded legs.  The map \(h_0\) agrees with \(h\) on \(\tropC_0\)
and has the same slope as \(h\) along each of these adjacent half-edges.

For a half-edge \(\zeta\) adjacent to a vertex \(V\) of $\tropC$, set
\[
d(\zeta,\tropC_0)
=
\operatorname{dist}(V,\tropC_0).
\]
If \(H\subseteq \R^2\) is an affine line, define
\[
\mathcal{F}_{H}(h)
=
\left\{
\zeta\ \middle|\
\begin{array}{l}
\zeta\text{ is a half-edge adjacent to a vertex }V,\\
h(V)\in H,\text{ and the germ }h(\zeta)\text{ is not contained in }H
\end{array}
\right\}.
\]
Note that when $\Gamma=\Gamma_d$, we have $\mathcal{F}_{H}(h) \neq \emptyset$ for all $H$.

\begin{definition}
A parametrized genus-one tropical curve
$
[h\colon \tropC\longrightarrow \R^2]
$
is \emph{well-spaced} if, for every affine line
\(H\subseteq \R^2\), one of the following holds:
\begin{enumerate}
    \item[$\bullet$] the neighborhood of the core is not contained in \(H\), that is,
    $
    h_0(\mathring C_0)\not\subseteq H;
    $
    \item[$\bullet$] the neighborhood of the core is contained in \(H\), and, whenever
    \(\mathcal{F}_{H}(h)\neq\varnothing\), the minimum
    \[
    d_H
    :=
    \min_{\zeta\in\mathcal{F}_{H}(h)}
    d(\zeta,C_0)
    \]
    is attained by at least three half-edges.
\end{enumerate}
\end{definition}

The moduli space of well-spaced tropical curves is the subcone complex
\[
\W_{\Gamma}(\R^2)
\subseteq
M^{\mathrm{trop,rad}}_{\Gamma}(\R^2)
\]
whose points parametrize well-spaced tropical curves.  It is a
pure-dimensional generalized cone complex. When $\Gamma=\Gamma_d$, its dimension is $n+3d=6d-1$ \cite[Theorem 3.2.10]{Tor14}. 

\begin{remark}
    As for the case of algebraic curves, we are not interested in the order of the markings $z_j$, so we will discard that in our computations.
\end{remark}

\subsubsection{Deficiency, overvalence and loop multiplicity}

Next we recall the notions of deficiency, overvalence and loop multiplicity of tropical curves.

\begin{definition} 
    For a parametrized tropical curve $[h \colon \tropC \to \R^2]$:
    \begin{enumerate}
        \item The deficiency of $[h \colon \tropC \to \R^2]$ is the codimension in $\R^2$ of the subspace spanned by the image of the cycle $h(\tropC_0)$.
        \item The overvalence of $\tropC$ is 
        \[
        \ov{\tropC} = \sum_{V \in V(\tropC),\, \val{V} \geq 3} (\val{V} - 3).
        \]
    \end{enumerate}
\end{definition}

\begin{definition}\label{def:loop}
The loop multiplicity of a tropical map $[h \colon \tropC \to \R^2]$ is defined as the index of the map
\[
\mathbb{Z}^{E(\tropC_0)} \longrightarrow \operatorname{span}_{\mathbb{R}}(h(\tropC_0)) \cap \mathbb{Z}^2,
\]
given by
\[
(\ell_1, \ldots, \ell_k) \longmapsto \sum \ell_i u_i.
\]
Here, $E(\tropC_0)$ denotes the edge set of $\tropC$, and, for each edge $e_i \in E(\tropC_0)$, the symbol $u_i$ denotes the direction vector corresponding to a chosen half-edge composing $e_i$.
\end{definition}

Deficiency, overvalence and loop multiplicity are independent of the choice of a tropical curve in the interior of a cone $\sigma$ of $\mathsf{W}_{\Gamma}(\mathbb{R}^2)$. We thus obtain well-defined notions $\mathrm{def}(\sigma)$, $\mathrm{ov}(\sigma)$, and $\mathrm{loop}(\sigma)$ for the deficiency, overvalence, and loop multiplicity of a cone $\sigma$.

\subsubsection{Multiplicities of maximal cones of $\W_{\Gamma}(\R^2)$}

\begin{remark}
    Table~\ref{tab: comparing weights deficiency 1} (resp. Table~\ref{tab: comparing weights deficiency 2}) reports the core $\tropC_0$ of all well-spaced tropical curves of deficiency $1$ (resp. deficiency $2$) that will appear in our enumerations (where some marked points $p_i$ may be missing). See Propositions \ref{prop: contracted edge} and \ref{prop: curves-small-j}. 
\end{remark}

Let $\sigma$ be a maximal cone of $W_{\Gamma}(\mathbb{R}^2)$. To each such cone, Cela--Koyama \cite{CK} associate a multiplicity $m(\sigma)$. For the cones appearing in our enumeration, these multiplicities are defined as follows:
\begin{itemize}
    \item If $\mathrm{def}(\sigma) = 0$, then $m(\sigma) = \mathrm{loop}(\sigma)$.
    \item If $\mathrm{def}(\sigma) = 1$ and the core $\tropC_0$ is as shown in Table~\ref{tab: comparing weights deficiency 1}, then $m(\sigma)$ is given in the \cite{CK} row of that table.
    \item If $\mathrm{def}(\sigma) = 2$ and the core $\tropC_0$ is as shown in Table~\ref{tab: comparing weights deficiency 2}, then $m(\sigma)$ is given in the \cite{CK} row of that table.
\end{itemize}

\begin{remark}
In the notation of \cite{CK}, the multiplicity $m(\sigma)$ denotes the product $\nn(\sigma) \cdot \mathrm{loop}(\sigma)$, where $\nn(\sigma)$ represents the automorphism-weighted number of algebraic lifts of the tropical curves parametrized by $\sigma$. As shown in \cite[Theorem C]{CK}, the quantity $\nn(\sigma)$ equals the number of interior lattice points of a polygon constructed from the neighborhood of $\tropC_0$, divided by the order of the automorphism group of the maps parametrized by $\sigma$. We will not require the explicit construction of $\nn(\sigma)$ directly, as we only use $m(\sigma)$, which in dimension $2$ simplifies to the values given above.

\end{remark}

\subsubsection{The $\mathrm{ev} \times j$-determinant}

The moduli space $\W_{\Gamma}(\R^2)$ comes with a natural morphism

$$
\mathrm{ev} \times j: \W_{\Gamma}(\R^2) \to (\R^2)^n \times M_{1,1}^{\mathrm{trop}}
$$
where $M_{1,1}^{\mathrm{trop}} \simeq \R_{\geq 0}$ is the moduli space of tropical $1$-marked elliptic curves. The morphism $\mathrm{ev} \times j$ is given by the evaluation at the marking $p_i$ for $i=1,\ldots,n$ and the stabilized domain graph with the first marking.

Assume that source and target of the morphism $\mathrm{ev} \times j$ have the same dimension. Each maximal cone $\sigma \subset \W_{\Gamma}(\mathbb{R}^2)$ possesses a natural integral structure given by a lattice $N_\sigma$ such that$$\sigma \subseteq N_\sigma \otimes_{\mathbb{Z}} \mathbb{R}.$$
More precisely, by recording the lengths of all edges in $\tropC$, $\sigma$ embeds naturally into $\mathbb{R}_{\ge 0}^{E(\tropC)}$ as the solution set defined by loop-closure and well-spacedness conditions. The lattice is then defined as
$$
N_\sigma := \mathrm{Span}_{\mathbb{R}}(\sigma) \cap \mathbb{Z}^{E(\tropC)}.
$$
The restriction of the $\mathrm{ev} \times j$ map to $\sigma$ extends to a linear map 
$$
(\mathrm{ev} \times j )(\sigma): N_\sigma  \to (\Z^2)^n \times \Z
$$
\begin{definition}
    The $\mathrm{ev} \times j$-determinant of $\sigma$, denoted by $\mathrm{det}_{\mathrm{ev} \times j}(\sigma)$ is the lattice of the image of the map $(\mathrm{ev} \times j )(\sigma)$.
\end{definition}

\subsection{Correspondence theorems}

In this paper, we will use two different correspondence theorems: the genus $0$ correspondence for point insertions in $\mathbb{P}^2$ due to Mikhalkin \cite{Mik05}, and the genus $1$ correspondence for point insertions in $\mathbb{P}^2$ with fixed $j$-invariant due to Cela-Koyama \cite{CK}. Below, we recall the necessary notation and state both results.

\subsubsection{Genus $0$ correspondence}

Mikhalkin \cite{Mik05} established a correspondence theorem equating the number $N_d$ of planar rational curves of degree $d$ passing through $3d-1$ general points with the weighted count $N_d^{\mathrm{trop}}$ of rational tropical curves in $\mathbb{R}^2$ of degree $\delta_d$ passing through $3d-1$ points in general position in $\mathbb{R}^2$; that is, $N_d = N_d^{\mathrm{trop}}$. The degree $N_d$ was computed by Kontsevich \cite{Kont} using Gromov--Witten theory, while the tropical count $N_d^{\mathrm{trop}}$ was undertaken by Gathmann--Markwig \cite{GM08}.

A genus $0$ tropical curve $h': \tropC' \to \mathbb{R}^2$ contributing to Mikhalkin's count and lying in the interior of a maximal cone $\sigma'$ of $M_{\Gamma_d'}(\mathbb{R}^2)$ contributes with weight
\[
\mathrm{det}_{\mathrm{ev}}({\sigma'}).
\]
Here $\mathrm{ev} \colon M_{\Gamma_d'}(\mathbb{R}^2) \to (\mathbb{R}^2)^{3d-1}$ denotes the tropical evaluation map. As in the case of well-spaced elliptic curves, $\mathrm{ev}$ restricts on each cone $\sigma'$ of $M_{\Gamma_d'}(\mathbb{R}^2)$ to a linear map. We thus obtain a linear map of lattices, whose index is what we denoted by $\mathrm{det}_{\mathrm{ev}}({\sigma'})$. Note that $M_{\Gamma_d'}(\mathbb{R}^2)$ is a pure dimensional generalized cone complex of dimension $6d-2$.

An important feature in genus $0$ is that all tropical curves $h' \colon \tropC' \to \mathbb{R}^2$ contributing to $N_d^{\mathrm{trop}}$ for points $x_i$ in general position are simple \cite[Proposition 4.11]{Mik05}. This concept is made precise in the following definitions.

\begin{definition} \label{def: simple} \cite[Definition 4.2]{Mik05}
A genus-$0$ parametrized tropical curve $h \colon \tropC' \to \mathbb{R}^2$ is called simple if it satisfies all of the following conditions:
\begin{enumerate}[label=(\roman*)]
    \item The underlying graph of $\tropC'$ is trivalent.
    \item The map $h$ is an immersion.
    \item For any point $y \in \mathbb{R}^2$, the fiber $h^{-1}(y)$ contains at most two points.
    \item If $a, b \in \tropC'$ with $a \neq b$ satisfy $h(a) = h(b)$, then neither $a$ nor $b$ is a vertex of $\tropC'$.
\end{enumerate}
\end{definition}

\begin{definition}\label{def: general}\cite[Definition 4.7]{Mik05}
A set of points $x_1, \dots, x_{3d-1} \in \mathbb{R}^2$ is said to be in general position if every tropical curve $h \colon \tropC' \to \mathbb{R}^2$ of genus $0$ with $m$ ends passing through $x_1, \dots, x_{3d-1}$ with $3d \ge  m$ satisfies the following conditions:
\begin{enumerate}[label=(\roman*)]
    \item The curve $h' \colon \tropC' \to \mathbb{R}^2$ is simple.
    \item The inverse images $(h')^{-1}(x_1), \dots, (h')^{-1}(x_k)$ are disjoint from the vertices of $\tropC$.
    \item $3d = m$.
\end{enumerate}
\end{definition}

By \cite[Proposition~4.11]{Mik05}, the configurations of points in $(\mathbb{R}^2)^{3d-1}$ in tropically general position form a dense open subset.

Note that combining the final conditions of Definitions~\ref{def: simple} and~\ref{def: general}, we see that for every curve enumerated in $N_d^{\mathrm{trop}}$, the directions of the three half-edges meeting at any vertex span $\mathbb{R}^2$.

\subsubsection{Genus $1$ correspondence}

In \cite{CK}, the authors established a correspondence theorem for elliptic curves in any toric varieties.  For planar ellitpic curves with fixed $j$ invariant this reduces to the following.

\begin{theorem}{\cite[Corollary 1.7]{CK}}
    Let $x_1,\ldots,x_{3d-1} \in \R^2$  be points in general position and $j \in \R_{\geq 0}$ be general. Then, we have 
    $$
    E_{d,j}= \sum_{\sigma} m(\sigma) \cdot \mathrm{det_{ev \times j}}(\sigma)
    $$
    where the sum is over the maximal cones $\sigma$ in $\W_{\Gamma_d}(\R^2)$ containing a map $[h]$ such that $h(p_i)=x_i$ for all $i=1,\ldots,3d-1$ and whose stabilized domain curve in $M_{1,1}^{trop}$ has cycle length $j$.
\end{theorem}

\section{Tropical enumerations}

In this section, we compute the fixed $j$-invariants $E_{d,j}$, relating them to Kontsevich's count of rational curves. We do this by considering the limits where $j$ is very large and where $j$ is very small. These different approaches 
yield distinct, interesting formulas.

\subsection{The case of large $j$ invariant}\label{sec: j large}

Fix points $x_1, \ldots, x_{3d-1} \in \R^2$ in general position, and fix $j \in \R_{\geq 0}$ to be very large.

\begin{proposition}\label{prop: contracted edge}
    Let $h: \tropC \rightarrow \mathbb{R}^2$ be a well-spaced genus $1$ parametrized tropical curve of degree $\delta_d$ passing through the points $x_i$ and with cycle length $j$. Let $\sigma$ be the maximal cone in $\W_{\Gamma}(\mathbb{R}^2)$ which contains $[h]$. If the $(\mathrm{ev} \times j)$-determinant of $\sigma$ is non-zero, then the cycle contains a single contracted bounded edge. Furthermore, one of the following holds:
    \begin{itemize}
        \item $\defic{\sigma} = 0$, and the contracted bounded edge is adjacent to two trivalent vertices; or
        \item $\defic{\sigma} = 2$, and the contracted bounded edge forms a self-loop at a $4$ or $5$-valent vertex. Moreover, either:
        \begin{itemize}
            \item the loop in $\tropC$ is based at a $5$-valent vertex; or
            \item the loop is based at a $4$-valent vertex $V$, and $V$ has two adjacent edges not in the loop of equal length.
        \end{itemize}
    Both cases are illustrated in Table~\ref{tab: comparing weights deficiency 2}.
        \end{itemize}
\end{proposition}

\begin{remark}\label{rmk: pb-string}
    A version of the above proposition appeared in \cite[Lemma 6.1]{KM09}, but their argument does not apply verbatim here, as the existence of what the authors call a `string' does not interact well with cones satisfying a non-trivial well-spacedness condition.
\end{remark}

\begin{proof} 
    We split the proof into cases depending on the deficiency of $\sigma$.
    
    If $\defic{\sigma} = 0$, then for dimension reasons, the curve is trivalent. By \cite[Proposition 6.1]{KM09} (our multiplicities agree in this case) there is a contracted bounded edge which is part of the cycle. We obtain the first case. 

    If $\defic{\sigma} = 1$, then we can find a genus-$0$ tropical curve $h' \colon \tropC' \to \mathbb{R}^2$ with the same image as $h \colon \tropC \to \mathbb{R}^2$, thereby passing through the same point conditions. If $h$ maps the cycle in $\tropC$ to the image of a bounded edge of $\tropC'$, then the total length of the cycle is bounded above by twice the maximum edge length of $\tropC'$, unless there exists a contracted edge in the cycle. Since $\defic{\sigma} > 0$, any such contracted edge must connect two distinct edges. This would force $h'$ to feature either a $4$-valent vertex or two parallel edges, violating condition (i) or (iv) in Definition~\ref{def: simple}. Otherwise, $h$ maps the cycle of $\tropC$ to the image of an infinite end of $\tropC'$. However, this is impossible because all direction vectors in $\delta_d$ are primitive.
    
    If $\defic{\sigma} = 2$, the the cycle is contracted, and so there exists a contracted edge. By dimension considerations, there cannot be more than one contracted edge, so the cycle consists of a self-loop at a $4$ or $5$-valent vertex. The condition on edge lengths arises from well-spacedness.
\end{proof}

By removing the contracted edge (and removing any 2-valent vertices that appear), we obtain a genus $0$ parametrized tropical curve which we denote $h': \tropC' \rightarrow \mathbb{R}^2$, also passing through $x_1, \ldots, x_{3d-1}$.

The first part of Theorem \ref{thm: main} will follow from the following proposition, which the rest of the section will be spent proving. 

\begin{proposition} \label{prop:largej}
    Fix a genus $0$ parametrized tropical curve $h': \tropC' \rightarrow \mathbb{R}^2$ passing through $x_1, \ldots, x_{3d-1}$, and let $\sigma'$ be the maximal cone in $M^{\mathrm{trop}}_{\Gamma'_d}(\mathbb{R}^2)$ which contains $[h']$. Let $\mathcal{S}$ be the collection of maximal cones in $W_{\Gamma_d}(\mathbb{R}^2)$ corresponding to genus $1$ parametrized tropical curves passing through $x_1, \ldots, x_{3d-1}$ with cycle length $j$ such that $h': \tropC' \rightarrow \mathbb{R}^2$ is obtained removing the contracted edge. Then 
    $$\sum_{\sigma \in \mathcal{S}} m(\sigma) \cdot \mathrm{det_{ev \times j}}(\sigma) = {d-1 \choose 2} \mathrm{det_{ev}}(\sigma'),$$ 
\end{proposition}

To prove Proposition \ref{prop:largej}, we will distinguish the two cases in Proposition~\ref{prop: contracted edge}. 

Let $\sigma'$ and $\mathcal{S}$ be as in Proposition~\ref{prop:largej}. Write $\mathcal{S} = \mathcal{S}_0 \sqcup \mathcal{S}_2$, where $\mathcal{S}_i$ is the collection of maximal cones $\sigma \subset W_{\Gamma_d}(\mathbb{R}^2)$ corresponding to genus-one parametrized tropical curves passing through $x_1, \ldots, x_{3d-1}$ with cycle length $j$ and deficiency $i$, such that cutting the contracted edge in the curves parametrized by $\sigma$ yields curves in $\sigma'$.

\subsubsection{Contribution from curves of deficiency $2$ to the large $j$ enumeration}

First, suppose that $\sigma \in \mathcal{S}_2$.
To state our next result we require the following standard notations.

\begin{definition}
        The \emph{weight} $w(e)$ of $e$ is the largest integer such that the direction vector $u_e$ of $e$ can be written as $u_e = w(e) u'$ for $u' \in \mathbb{Z}^2$. 
    \end{definition}

    \begin{definition}
        Given a lattice polytope $P \subset \mathbb{Z}^2$, we write $i(P)$ for the number of interior lattice points, and $b(P)$ for the number of points on the boundary. 
        
    \end{definition}

    Finally, \cite{Mik05} associated to every simple genus-$0$ tropical curve a subdivision of the triangle $\Delta_d$ with vertices $(0,0)$, $(0,d)$, and $(d,0)$.

    \begin{lemma} \label{lem:def2}
    We have
    $$\sum_{\sigma \in \mathcal{S}_2} m(\sigma) \cdot \mathrm{det_{ev \times j}}(\sigma) = \mathrm{det_{ev}}(\sigma') \bigg[\sum_{T} i(T) + \sum_{e \in E(\tropC')} (w(e) - 1)  \bigg],$$
    where the first sum goes over the triangles of the Newton subdivision associated to $h': \tropC' \rightarrow \mathbb{R}^2$. 
    \end{lemma}

    \begin{proof}
    Let $\sigma$ be a cone in $\mathcal{S}_2$. If the curves $h$ in $\sigma$ are as in the first dash of the second bullet of Proposition \ref{prop: contracted edge}, then $m(\sigma) = i(T)$ where $T$ is the dual triangle to the vertex $V$. Moreover, $\mathrm{det_{ev \times j}}(\sigma) = \mathrm{det_{ev}}(\sigma')$ (see the proof of \cite[Lemma 4.10]{KM09}).

    Suppose instead that $\sigma$ parametrized curves as in the second dash of the second bullet of Proposition \ref{prop: contracted edge}. Let $e$ be the edge in $h': \tropC' \rightarrow \mathbb{R}^2$ which is formed after removing the contracted edge and resulting bivalent vertex. Then $w(e)$ is also the weights of the direction vectors of the non-contracted edges adjacent to $V$, and so $m(\sigma) = (w(e) - 1)/2$. 

    We have
    $$\mathrm{det_{ev \times j}}(\sigma) = 2 \times \mathrm{det_{ev}}(\sigma').$$
    Indeed, let $e_1, e_2$ be the non-contracted edges adjcant to $V$. A lattice basis of $N_\sigma$ in this case is given by the edge lengths of the bounded edges which are not $e_1$ or $e_2$, and the length of $e_1$ (which by well-spacedness must be equal to the length of $e_2$). Use this basis to obtain a matrix representation of the linear map
    $$
    (\mathrm{ev} \times j )(\sigma): N_\sigma  \to (\Z^2)^n \times \Z. 
    $$
    Then the column corresponding to the contracted edge has one entry, which is the row corresponding to the cycle length. Deleting this row and column gives a matrix representation of $\mathrm{ev}(\sigma'): N_{\sigma'} \to (\Z^2)^n$, except the column corresponding to the length of $e_1$ is replaced by a column corresponding to the length of $e$. Then length of $e$ is twice the length of $e_1$, and so the claim follows.  

    \end{proof}

\subsubsection{ Contribution from curves of deficiency $0$ to the large $j$ enumeration}

Next we deal with cones in $\mathcal{S}_0$.

    \begin{lemma}
    We have
    $$\sum_{\sigma \in \mathcal{S}_0} m(\sigma) \cdot \mathrm{det_{ev \times j}}(\sigma) = \mathrm{det_{ev}}(\sigma') \cdot \sum_{P} \mathrm{Area}(P),$$
    where sum goes over the parallelograms of the Newton subdivision associated to $h': \tropC' \rightarrow \mathbb{R}^2$.
    \end{lemma}

    \begin{proof}
    For $\sigma \in \mathcal{S}_0$, we immediately have $m(\sigma) = \mathrm{loop}(\sigma)$. Let $V_1, V_2$ be the vertices in $\tropC$ adjacent to the contracted edge. Choose $u_i$ such that $\pm u_i$ are the direction vectors of the non-contracted edges adjacent to $V_i$, for $i = 1,2$.  Then it is shown in \cite[Lemma 4.11]{KM09} that 
    $m(\sigma) \cdot \mathrm{det_{ev \times j}}(\sigma) = |\det(u_1, u_2)| \cdot \mathrm{det_{ev}}(\sigma')$. The conclusion follows observing that if $P$ is the parallelogram in the Newton subdivision associated to the two crossing edges in $h': \tropC' \rightarrow \mathbb{R}^2$, then $|\det(u_1,u_2)|$ is the area of $P$. 

    \end{proof}

\subsubsection{Proof of Proposition \ref{prop:largej} and Theorem \ref{thm: main} for large $j$}

We now prove Proposition \ref{prop:largej}

\begin{proof}[Proof of of Proposition \ref{prop:largej}]
    We need to show that 
    $$\sum_{T} i(T) + \sum_{e \in E(\tropC')} (w(e) - 1) + \sum_{P} \mathrm{Area}(P) = {d-1 \choose 2},$$
    where $T$ (resp. $P$) runs over the triangles (resp. parallelograms) of the Newton subdivision associated to $h': \tropC' \rightarrow \mathbb{R}^2$. 

    Every bounded edge $e \in E(\tropC')$ is adjacent to two vertices in $V(\tropC')$, and the number of interior lattice points on the edge of the Newton subdivision corresponding to $e$ is $w(e)-1$. So 
    $$\sum_{e \in E(\tropC')} (w(e) - 1)  = \sum_T \frac{b(T) - 3}{2}.$$

    By Pick's theorem $\mathrm{Area}(P)  = i(P) + b(P)/2- 1.$ 
    Thus 
    \begin{align*}
        &\sum_{T} i(T) + \sum_{e \in E(\tropC')} (w(e) - 1) + \sum_{P} \mathrm{Area}(P) \\
        &= \sum_T i(T) + \sum_P i(P) + \sum_T \frac{b(T) - 3}{2} + \sum_P \bigg(\frac{b(P)}{2}-1\bigg) \\
        &= \sum_T i(T) + \sum_P i(P) + \sum_T \frac{b(T) - 3}{2} + \sum_P \frac{b(P) - 4}{2}+\#P,
    \end{align*}
    where $\#P$ denotes the number of parallelograms in the Newton subdivision.
    
    Now $\sum_T i(T) + \sum_P i(P)$ is the number of interior lattice points of $\triangle_d$ which are contained in the interior of a polygon of the subdivision. The sum $\sum_T \frac{b(T) - 3}{2} + \sum_P \frac{b(P) - 4}{2}$ is equal to the number of interior lattice points of $\triangle_d$ which are contained in the boundary of a polygon of the subdivision but are \emph{not} a vertex of any polygon. Finally, the number of parallelograms $\#P$ is equal to the number of interior lattice points that are vertices of the subdivision. Indeed, the difference between the number of such interior lattice points and the number of parallelograms equals the genus of the source curve, which here is $0$. This is because the image of the curve contains a loop for every such vertex, while each parallelogram corresponds to a node in the image not arising from the source curve.
    
    Concluding we have 
    \begin{align*}
        \sum_{T} i(T) + \sum_{e \in E(\tropC')} (w(e) - 1) + \sum_{P} \mathrm{Area}(P) = i(\Delta_d) = {d-1 \choose 2}. 
    \end{align*}
    
\end{proof}

\subsection{The case of small $j$ invariant} \label{sec: j small}
We now prove the second part of Theorem \ref{thm: main}. Fix points $x_1, \dots, x_{3d-1} \in \mathbb{R}^2$ in general position, and choose $j > 0$ to be sufficiently small. If necessary, we may restrict the points $x_i$ to vary within a smaller open set than the one specified by general position in Definition \ref{def: general}.

\begin{proposition}\label{prop: curves-small-j}
    Let $h: \tropC \rightarrow \mathbb{R}^2$ be a well-spaced genus $1$ well-spaced parametrized tropical curve of degree $\delta_d$ passing through the points $x_i$ and with cycle length $j$. Let $\sigma$ be the maxima lcone in $W_{\Gamma}(\mathbb{R}^2)$ which contains $[h]$. If the $(\mathrm{ev} \times j)$-determinant of $\sigma$ is non-zero, then oone of the following holds:
    \begin{itemize}
        \item $\defic{\sigma} = 0$, and the cycle is formed of three vertices and three edges (with no marked points). 
        \item $\defic{\sigma} = 1$, and there is a cycle formed by two vertices $V_1, V_2$ joined by two edges (no marked points are on the cycle). Moreover, we have one of two cases:
            \begin{itemize}
                \item One between $V_1, V_2$ is 4-valent and the other is 3-valent. The half-edge not in the cycle and adjacent to the 3-valent vertex in the cycle is not an infinite end.  
                \item Both $V_1, V_2$ are 3-valent. Let $e_1, e_2$ be the two half-edges adjacent to $V_1, V_2$ respectively which are not in the cycle. Then $e_1$ and $e_2$ are not infinite ends and they have the same length. 
            \end{itemize}
            Both cases are illustrated in Table \ref{tab: comparing weights deficiency 1}.
        \item $\defic{\sigma} = 2$, and there is a single contracted edge which forms a self-loop, at a $4$ or $5$-valent vertex. Such vertex carries no marking.
            Both cases are illustrated in Table \ref{tab: comparing weights deficiency 2}.
    \end{itemize}

\end{proposition}

\begin{proof}
    As $j$ goes to $0$ (keeping $x_1, \dots, x_{3d-1}$ fixed), there is a fixed collection of maximal cones of $\W_{\Gamma_d}(\mathbb{R}^d)$ contributing to $E_{d,j}^{\mathrm{trop}}$. When $j = 0$, each such cone contains a curve $h'$ with $j=0$ passing through the $x_i$, yielding a genus $0$ curve $h': \tropC' \to \mathbb{R}^2$. We split the proof into cases depending on the deficiency of $\sigma$.
    
    If $\defic{\sigma} = 0$, then for dimension reasons, the curve $\tropC$ is trivalent. When $j=0$, the curve $\tropC'$ is trivalent (as it still interpolates the points $x_i$), thus for the cycle in $\tropC$ to span $\R^2$ we need the it to have three vertices and three edges (with no marked points).

    If $\operatorname{def}(\sigma) = 1$, then again at $j = 0$, the curve $\tropC'$ is trivalent and interpolates the points $x_i$. The result follows. Note that this proves only that the cycle can carry at most one marking. However, because the points $x_i$ are general, one can in fact exclude this possibility as well. Finally, every infinite edge has primitive direction vector, and so if one of the vertices $V_i$ in the cycle is 3-valent, the half-edge adjacent to $V_i$ not contained in the cycle cannot be an infinite end. The condition on edge lengths in the second dash comes from well-spacedness.

    If $\defic{\sigma} = 2$, the the cycle is contracted, and so there exists a contracted edge. By dimension considerations, there cannot be more than one contracted edge, so the cycle consists of a self-loop at a $4$ or $5$-valent vertex, which carries no marking (otherwise, the vertex would have to be 5-valent, in which case either the curve could not be well-spaced or $V$ could not be balanced).   
\end{proof}

\begin{proposition} \label{prop:smallj}
    Fix a genus 0 curve  $h': \tropC' \rightarrow \mathbb{R}^2$ through $x_1, \ldots, x_{3d-1}$,  and let $\sigma'$ be the maximal cone in $M^{\mathrm{trop}}_{\Gamma'_d}(\mathbb{R}^2)$ which contains $[h']$. Let $\mathcal{S}$ be the collection of maximal cones in $W_{\Gamma_d}(\mathbb{R}^2)$ corresponding to genus $1$ parametrized tropical curves passing through the same conditions with cycle length $j$, such that the curve with $j = 0$ in the same cone is given by $h': \tropC' \rightarrow \mathbb{R}^2$ after replacing the genus 1 vertex with a genus 0 vertex, and removing the resulting vertex if it is 2-valent). Then 
    $$\sum_{\sigma \in \mathcal{S}} m(\sigma) \cdot \mathrm{det_{ev \times j}}(\sigma) = \mathrm{det_{ev}}(\sigma') \cdot \sum_T (2\mathrm{Area}(T)^2- \frac{1}{2}),$$
    where $T$ runs over triangles of the Newton subdivision associated to $h': \tropC' \rightarrow \mathbb{R}^2$. 
\end{proposition}

We postpone the proof of this proposition until later. In the proof, we will distinguish cases according to Proposition \ref{prop: curves-small-j}. Let $\sigma'$ and $\mathcal{S}$ be as in Proposition~\ref{prop:smallj}. Write $\mathcal{S} = \mathcal{S}_0 \sqcup \mathcal{S}_1 \sqcup \mathcal{S}_2$, where $\mathcal{S}_i$ is the collection of maximal cones $\sigma \subset W_{\Gamma_d}(\mathbb{R}^2)$ contained in $\mathcal{S}$ and of deficiency $i$.

\begin{lemma}
    We have
    $$\sum_{\sigma \in \mathcal{S}_0} m(\sigma) \cdot \mathrm{det_{ev \times j}}(\sigma) = \mathrm{det_{ev}}(\sigma') \cdot \bigg[\sum_T i(T) \cdot 2 \mathrm{Area}(T) \bigg]$$
    where T goes over all triangles in the Newton subdivision dual to $h': \tropC' \rightarrow \mathbb{R}^2$.
\end{lemma}

\begin{proof}
    In the deficiency $0$ case the multiplicities are the same as in \cite{KM09}, so the result follows from the proof of \cite[Lemma 7.1]{KM09}. 
\end{proof}

\begin{lemma} \label{lem: def1}
    We have
    $$\sum_{\sigma \in \mathcal{S}_1} m(\sigma) \cdot \mathrm{det_{ev \times j}}(\sigma) = \mathrm{det_{ev}}(\sigma') \cdot \bigg[\sum_T 
    (b(T)-3) \cdot \mathrm{Area}(T) \bigg]$$
    where T goes over all triangles in the Newton subdivision dual to $h': \tropC' \rightarrow \mathbb{R}^2$.
\end{lemma}

This is different from the corresponding deficiency $1$ discussion in \cite{KM09}.

\begin{proof}
    We split the discussion into cases as in the two dashes in the second bullet of Proposition \ref{prop: curves-small-j}. 
    
    Suppose we wish to construct from $h'$ a genus 1 parametrized tropical curve with small loop length $j$, of deficiency 1, where one of the vertices in the loop is 4-valent. There are several ways to do this. We first choose an edge $e$ in $E(\tropC')$, and a vertex $V \in V(\tropC')$ adjacent to $e$. Let $a, b \in \mathbb{Z}_{> 0}$ such that $a<b$ and $a + b = w(e)$. For each such $a,b$, we can form $h: \tropC \rightarrow \mathbb{R}^2$ in a unique way so that the cycle contains the vertex $V$, it has one edge of weight $a$ and another of weight $b$ and its image is contained in that of $e$. See the left hand side of Table \ref{tab: comparing weights deficiency 1}. 
    
    Let $\sigma$ be the corresponding cone in $\W_{\Gamma}(\mathbb{R}^2)$. Let $u$ be the primitive direction of edge $e$, and let $v$ be the direction vector of any other edge adjacent to $V$. 

    As explained in Table \ref{tab: comparing weights deficiency 1}, the multiplicity of $\sigma$ is 
    $$m(\sigma) = \begin{cases}
        \gcd(a,b) \cdot (|\det(u,v)| - 1) \quad \text{if } a \neq b \\
        \gcd(a,b) \cdot (|\det(u,v)| - 1)/2 \quad \text{if } a = b.
    \end{cases}$$
    Let $a', b'$ be such that $a = \gcd(a,b)\cdot a'$ and $b = \gcd(a,b) \cdot b'$. From the proof of \cite[Lemma 4.12]{KM09}, we have 
    $$
    \mathrm{det_{ev \times j}}(\sigma) = 
        (a' + b') \cdot \mathrm{det_{ev}}(\sigma') 
    $$
    If $w(e)$ is odd, the contribution to the sum in the statement coming from our choice of $V$ and $e$ is
    \begin{align*}
        & \mathrm{det_{ev}}(\sigma') \sum_{0<a< b, a+b = w(e)} \gcd(a,b) \cdot (|\det(u,v)| - 1) \cdot (a' + b') \\
        &=\mathrm{det_{ev}}(\sigma')  \sum_{0<a< b, a+b = w(e)}  (|\det(u,v)| - 1) \cdot w(e)  \\
        &= \mathrm{det_{ev}}(\sigma')  \sum_{0<a< b, a+b = w(e)} \bigg( 2 \mathrm{Area}(T) - w(e) \bigg)\\
        &= \mathrm{det_{ev}}(\sigma')  \cdot \frac{w(e) - 1}{2} \cdot \bigg( 2 \mathrm{Area}(T) - w(e) \bigg)
    \end{align*} 
    where $T$ is the triangle associated to $V$ in the Newton subdivision of $h': \tropC \rightarrow \mathbb{R}^2$. 

    If instead $w(e)$ is even, the total contribution is 
     \begin{align*}
        & \mathrm{det_{ev}}(\sigma') \bigg[\sum_{0<a< b, a+b = w(e)} \bigg(\gcd(a,b) \cdot (|\det(u,v)| - 1) \cdot (a' + b') \bigg) \\
        &\hspace{5cm} + \frac{\mathrm{gcd}(\frac{w(e)}{2},\frac{w(e)}{2}) \cdot (|\det(u,v)| - 1) (1+1)}{2}  \bigg] \\
        &=\mathrm{det_{ev}}(\sigma') \bigg[ \cdot \frac{w(e) - 2}{2} \cdot \bigg( 2 \mathrm{Area}(T) - w(e) \bigg) +\frac{1}{2}\bigg( 2 \mathrm{Area}(T) - w(e) \bigg)\bigg] \\
        &= \mathrm{det_{ev}}(\sigma')  \cdot \frac{w(e) - 1}{2} \cdot \bigg( 2 \mathrm{Area}(T) - w(e) \bigg)
    \end{align*} 

    Suppose now we wish to construct a deficiency 1 genus 1 parametrized tropical curve $h: \tropC \rightarrow \mathbb{R}^2$ where both the vertices in the loop are 3-valent. There is one such construction for every $e$ in $E(\tropC')$ and $a,b \in \mathbb{Z}_{>0}$ with $a<b$ and $a+b = w(e)$. Let $\sigma$ again be the corresponding cone in $\W_\Gamma(\mathbb{R}^2)$. As explained in Table \ref{tab: comparing weights deficiency 1}, the multiplicity of $\sigma$ is 
    $$m(\sigma) = \begin{cases}
        \gcd(a,b) \quad \text{if } a \neq b \\
        \gcd(a,b)/2 \quad \text{if } a = b.
    \end{cases}.$$
    Let $a', b'$ be such that $a = \gcd(a,b)\cdot a'$ and $b = \gcd(a,b) \cdot b'$. We claim that 
    $$
    \mathrm{det_{ev \times j}}(\sigma) = 
        2(a' + b') \cdot \mathrm{det_{ev}}(\sigma').  
    $$
    Indeed, let $e_1, e_2 \in E(\tropC)$ be the edges adjacent to the vertices $V_1, V_2$ (the vertices of the cycle), not contained in the cycle. A lattice basis of $N_\sigma$ in this case is given by the edge lengths of the bounded edges which are not $e_1$ or $e_2$, the length of $e_1$ (which by well-spacedness must be equal to the length of $e_2$), and a lattice basis of $\{(x,y) \in \mathbb{Z}^2| ax = by\}$. The last element comes from the lengths of the two edges in the cycle, with the condition that the cycle must close. Explicitly, we take the element $X = (b', a') \in \mathbb{Z}^2$. 
    
    Use this basis to obtain a matrix representation of the linear map
    $$
    (\mathrm{ev} \times j )(\sigma): N_\sigma  \to (\Z^2)^n \times \Z. 
    $$
    The row corresponding to $j$ consists of one entry $a' + b'$ in the column corresponding to $X$. Deleting this row and column gives a matrix representation of $\mathrm{ev}(\sigma'): N_{\sigma'} \to (\Z^2)^n$, except the column corresponding to the length of $e_1$ is replaced by a column corresponding to the length of $e$. The former column is twice the latter column, and the claim follows. 

    If $w(e)$ is odd, the corresponding contribution is 
    \begin{align*}
        & \mathrm{det_{ev}}(\sigma') \sum_{0<a< b, a+b = w(e)} \gcd(a,b) \cdot 2 (a' + b') \\
        &=\mathrm{det_{ev}}(\sigma')  \cdot (w(e)-1)\cdot w(e). 
    \end{align*} 
    A similar calculation gives the same result for $w(e)$ odd.

    Now, let $e$ be a bounded edge of $\tropC'$ with adjacent vertices $V_1$ and $V_2$, corresponding to triangles $T_1, T_2$ in the Newton subdivision. By the above discussion, the contribution to the sum in the statement of the lemma coming from the edge $e$ and the pairs $(V_1,e)$ and $(V_2,e)$ is
    $$\mathrm{det_{ev}}(\sigma')  \cdot \frac{w(e) - 1}{2} \cdot  2( \mathrm{Area}(T_1 ) +  \mathrm{Area}(T_2 )).$$
    Summing over every edge of $\tropC'$ (including unbounded ones where anyway $w(e)-1$ vanishes) one obtains  
    \begin{align*}
        & \mathrm{det_{ev}}(\sigma')  \sum_{V \in V(\tropC')}  \bigg[\sum_{e \text{ adjacent to $V$}}(w(e) - 1) \bigg] \cdot  \mathrm{Area}(T_V) 
    \end{align*}
    where $T_V$ denotes the triangle in the Newton subdivision of $h'$ corresponding to $V$.
    The result follows since $\sum_{e \text{ adjacent to $V$}}(w(e) - 1) = b(T_V)-3$. 
\end{proof}

\begin{lemma}
    We have
    $$\sum_{\sigma \in \mathcal{S}_2} m(\sigma) \cdot \mathrm{det_{ev \times j}}(\sigma) = \mathrm{det_{ev}}(\sigma') \cdot \bigg[\sum_T 
    \big(\mathrm{Area}(T) - \frac{1}{2} \big)\bigg]$$
    where T goes over all triangles in the Newton subdivision dual to $h': \tropC' \rightarrow \mathbb{R}^2$.
\end{lemma}

\begin{proof}
    By Lemma \ref{lem:def2}, we have
    $$\sum_{\sigma \in \mathcal{S}_2} m(\sigma) \cdot \mathrm{det_{ev \times j}}(\sigma) = \mathrm{det_{ev}}(\sigma') \cdot \bigg[\sum_{T} i(T) + \sum_{e \in E(\tropC')} (w(e) - 1)\bigg].$$
    Noticing that $\sum_{e \in E(\tropC')} (w(e) - 1) = \sum_{T} (b(T)-3)/2$, the result follows by Pick's theorem. 
    
\end{proof}

We can finally prove Proposition \ref{prop:smallj} and the rest of Theorem \ref{thm: main}.

\begin{proof}[Proof of Proposition \ref{prop:smallj}]
    Summing the contributions from $\mathcal{S}_0, \mathcal{S}_1$, and $\mathcal{S}_2$ we obtain (similarly to the proof of \cite[Lemma 7.1]{KM09})
    \begin{align*}
        & \sum_{\sigma \in \mathcal{S}} m(\sigma) \cdot \mathrm{det_{ev \times j}}(\sigma) \\
        &= \mathrm{det_{ev}}(\sigma') \cdot \bigg[\sum_T i(T) \cdot 2 \mathrm{Area}(T) + \sum_T (b(T)-3) \cdot \mathrm{Area}(T)  + \sum_T 
    \big(\mathrm{Area}(T) - \frac{1}{2} \big)\bigg] \\
        &= \mathrm{det_{ev}}(\sigma') \cdot  \sum_T \bigg(2\big(i(T)+ \frac{b(T)}{2} -1\big)\cdot \mathrm{Area}(T) -  \frac{1}{2} \bigg) \\
        &= \mathrm{det_{ev}}(\sigma') \cdot \sum_T \bigg(2\mathrm{Area}(T)^2 -  \frac{1}{2} \bigg), 
    \end{align*}
    where the last line uses Pick's theorem. 
\end{proof}

\begin{proof}[End of proof of Theorem \ref{thm: main}]
Theorem \ref{thm: main} follows immediately from Proposition \ref{prop:smallj} after summing over all rational tropical curves $\tropC'$ through the $3d-1$ points $x_i$.
    
\end{proof}

\bibliographystyle{amsalpha}
\bibliography{library}

$\,$\
\noindent

$\,$\
\noindent
\textsc{Department of Pure Mathematics {\it \&} Mathematical Statistics, 
University of Cambridge, Cambridge, UK}

\textit{e-mail address:} \href{mailto:ac2758@cam.ac.uk}{ac2758@cam.ac.uk}

$\,$\
\noindent

$\,$\
\noindent
\textsc{Department of Pure Mathematics {\it \&} Mathematical Statistics, 
University of Cambridge, Cambridge, UK}

\textit{e-mail address:} \href{mailto:sk2050@cam.ac.uk}{sk2050@cam.ac.uk}

\end{document}